\documentclass[10pt,a4paper]{amsart}

\usepackage{lmodern}
\usepackage[utf8]{inputenc}
\usepackage[T1]{fontenc}

\usepackage{fancyref}
\usepackage[colorlinks=true]{hyperref}

\usepackage{amsmath}
\usepackage{amssymb}
\usepackage{amsthm}

\newtheorem{theo}{Theorem}
\newtheorem{prop}[theo]{Proposition}
\newtheorem{coro}[theo]{Corollary}

\allowdisplaybreaks

\newcommand{\kk}[1]{\mathbf{#1}}

\def\CC{\mathbf{C}}

\def\<{\langle}
\def\>{\rangle}
\def\fs{b}

\def\ton{\alpha}
\def\ttw{\beta}
\def\tth{\gamma}
\def\tfo{\delta}

\def\qandq{\quad\text{and}\quad}
\def\ov#1{\overline{#1}}

\DeclareMathOperator{\GL}{GL}

\def\hsc{holomorphic sectional curvature}
\def\bl#1{\widehat{#1}}
\def\blX{\bl{X}}

\author{Gunnar \TH\'or Magn\'usson}
\address{Hafnarfj\"or\dh{}ur, Iceland}
\email{gunnar@magnusson.io}
\date{\today}
\title{On the curvature of the blowup of a point}

\hypersetup{
 pdfauthor={Gunnar Þór Magnússon},
 pdftitle={On the curvature of the blowup of a point},
 pdfkeywords={},
 pdfsubject={},
 pdflang={English}}

\begin{document}

\begin{abstract}
We consider the blowup of a point of a compact K\"ahler manifold and a metric
of the form $\mu^*h + t \fs$ on it, where $h$ is a K\"ahler metric on the
original manifold and $\fs$ is Hermitian form that looks like the
Fubini--Study metric near the exceptional divisor.
We calculate the curvature tensor of this metric on the exceptional divisor
and show that its holomorphic sectional curvature is negative in some
directions for all small enough $t$, which torpedos a natural approach to
showing that blowups of manifolds of positive holomorphic sectional curvature
have positive curvature.
\end{abstract}

\maketitle

\section*{Introduction}

Let $X$ be a compact K\"ahler manifold of dimension $n$.
Let $h$ be a K\"ahler metric on $X$.
We write $D$ for its Chern connection and
$R(\alpha,\ov\beta,\gamma,\ov\delta) = h(\frac i2
D^2_{\smash{\alpha,\ov\beta}}\gamma, \ov\delta)$ for its curvature tensor. The
holomorphic sectional curvature of $h$ is
$$
H(\xi)
= R(\xi, \ov\xi, \xi, \ov\xi)/|\xi|^4,
$$
where $\xi$ is a nonzero tangent field.
We say that $h$ has positive holomorphic sectional curvature if $H > 0$ for all
tangent fields.
Examples of manifolds that carry such metrics are the complex projective space
with its Fubini--Study metric, the standard metric on a Grassmannian manifold,
and projectivized bundles over bases of positive holomorphic sectional curvature
~\cite{alvarez2018projectivized}.
Tsukamoto~\cite{tsukamoto1957kahlerian} showed that manifolds that admit such
metrics are simply connected, and recently Xiaokui Yang~\cite{yang2017rc}
proved that such manifolds are projective and rationally connected, answering a
question of Yau~\cite[Problem~67]{yau1993open}.

A related question of Yau (again Problem~67) is whether the blowup of a compact
K\"ahler manifold of positive \hsc{} along a smooth submanifold again has
positive \hsc.
This is open; even for complex surfaces it is not known whether del~Pezzo
surfaces admit positively curved metrics (aside from the del~Pezzo surfaces of
degrees 8 and 9, where the nontrivial one admits such a metric because it's
also a Hirzebruch surface and thus a projectivized bundle). In this note we
explain that the brute force approach to this problem does not work.

Prior art in this direction has focused on showing that projective bundles
over manifolds with positive \hsc{} are also positive, along with more
general fibrations with positively curved fibers and base.
The common thread in the papers of
Hitchin~\cite{hitchin1975curvature},
Sung~\cite{sung1997kahler},
\'Alvarez~\cite{alvarez2016positive},
\'Alvarez,
Heier and Zheng~\cite{alvarez2018projectivized} or Chaturvedi and
Heier~\cite{chaturvedi2020hermitian}
is to consider metrics of the form $\mu^* h + t \fs$, where $h$ is a metric
on the base and $\fs$ a fiberwise metric on the fibers and show they have
positive curvature for small enough $t$.
In this note we try this approach for the blowup of a point.
We calculate the full curvature tensor of $h_t$ on a point on the exceptional
divisor and show that the metric
does not have positive \hsc{} there, no matter what metric $h$ we started with.
This of course doesn't mean the blowup doesn't have positive \hsc{}, only that
the most straightforward method of constructing a metric on it fails to show that.

\section{Curvature}

Let $X$ be a compact K\"ahler manifold of dimension $\dim_{\CC} X = n$
and let $h$ be a K\"ahler metric on it.

\begin{prop}
\label{prop:fs}
Let $\mu : \bl X \to X$ be the blowup of $X$ at a point $p$.
There exists a closed $(1,1)$-form $\beta$ on $\blX$ whose restriction to
the exceptional divisor is the Fubini--Study metric.
\end{prop}

\begin{proof}
This is proved for blowups of smooth submanifolds in
Voisin's textbook~\cite{voisin2002theorie}.
We can simplify that proof a little since we're only blowing up points.
In that case the statement goes back at least to Kodaira's proof of the
embedding theorem.

Locally around $p$, which may assume is the origin in a complex vector space
$V$, the blowup is
$$
\bl V
= \{ (v,[w]) \in V \times \kk P(V) \mid v \in \CC w \}
$$
and $\mu$ is the projection onto the first factor.
Pick an inner product on $V$ and let $B(r) \subset V$ be a ball of radius $r$
centered at $0$.
We pick $r$ so that $B(r)$ fits into the coordinate chart that implicitly lurks
in the background.
Let $\psi$ be a bump function supported on that chart that is identically $1$
on $B(r)$.
The $(1,1)$-form $\frac i2 \partial\bar\partial \log |v|^2$ on $V \setminus
\{0\}$ descends to $\kk P(V)$ and defines the Fubini--Study metric.
If $p_j : V \times V \setminus \{0\} \to V$ for $j = 1,2$ are the projections
onto the first and second factors then
$\frac i2 \partial \bar\partial (p_1^*\psi \log |p_2^*v|^2)$
defines a closed $(1,1)$-form on $V \times \kk P(V)$ that restricts to the
pullback of the Fubini--Study metric by $p_2$ on $B(r) \times \kk P(V)$.
It also extends to the rest of $X$ by zero.
Its restriction to $\bl V$ is the form we want.
\end{proof}

Let $p \in X$ be a point and blow it up to obtain $\mu : \bl X \to X$.
Let $\fs$ be the Hermitian form associated to the $(1,1)$-form $\beta$ on $\bl X$
we constructed in Proposition~\ref{prop:fs}.
Then $h_t = \mu^*h + t \fs$ is a K\"ahler metric on $\bl X$ for all $t$ small
enough.
We're going to calculate its curvature tensor at a point on the exceptional
divisor.

Recall that locally around $p$ the blowup is
$$
\bl V
= \{ (v,[w]) \in V \times \kk P(V) \mid v \in \CC w \}.
$$
If $f \in \GL V$ then $f$ acts on the blowup by $f(v, [w]) = (f(v), [f(w)])$.
This is an isomorphism that maps the exceptional divisor to itself, and we can
map any point on the divisor to any other point on it.
Let $(0, [w])$ be a point on $E$.
Let's choose normal coordinates $(x_1,\ldots,x_n)$ centered at $p$.
There exists $f \in U(n)$ so that $f(w) = (0 \ldots, 0, 1)$, and the
coordinates obtained by applying $f$ to the old ones are still centered at $p$
and are normal there because $f \in U(n)$.
Picking the chart $\{(y_1, \ldots, y_n) \in \CC^n \mid y_n \not= 0 \}$ for
$\kk P(V)$
we realize the blowup as
$$
\bl X
= \{ (x,y) \in \CC^n \times \CC^{n-1}
\mid x_j y_k = x_k y_j \text{ for $j,k = 1,\ldots,n$, where $y_n = 1$}  \}.
$$
In these coordinates the point we want to calculate the \hsc{} at is $(0,0)$
and we have $D_{h,\xi} \partial / \partial x_j = 0$ at $0$ for $j = 1, \ldots, n$.

We also note that close to the exceptional divisor, $h_t = \mu^* h + t \fs$ is
just the restriction of the product metric $p_1^* h \oplus t p_2^* \fs$ on
$V \times \kk P(V)$ to $\bl X$, where $p_j$ are the projections onto the
factors and $\fs$ is the Fubini--Study metric.

The only equations that give us any information at $(0,0)$ are
$$
x_j - y_j x_n = 0, \quad j = 1, \ldots, n-1.
$$
Their differentials are
$$
dx_j - y_j dx_n - x_n dy_j = 0, \quad j=1,\ldots,n-1
$$
and so the tangent fields
$$
\xi_j = x_n e_j + f_j,
\quad j=1,\ldots,n-1,
\qandq
\xi_n = \sum_{j=1}^{n-1} y_j e_j + e_n
$$
are a basis for
the intersection of all the kernels of the differentials, that is, of $T_{\blX}$.
Here $e_j$ is the tangent field corresponding to the coordinate $x_j$
and $f_k$ the one corresponding to $y_k$.
Note that the orthogonal complement of $T_{\smash{\blX}}$ at $(0,0)$ is spanned by
$(e_1, \ldots, e_{n-1})$.
Any holomorphic tangent field $\xi$ on $\blX$ near $(0,0)$ can be written as
$$
\xi = \sum_{j=1}^n \ton_j \xi_j
= \sum_{j=1}^{n-1} (\ton_j x_n + \ton_n y_j)  e_j
+ \sum_{j=1}^{n-1} \ton_j f_j
+ \ton_n e_n,
$$
where the $\ton_j$ are holomorphic functions.

While writing out the curvature tensor it will be useful to have the
semipositive Hermitian form $\tau(\ton, \bar \ttw) = \<\ton, \bar \ttw\> -
\ton_n \bar \ttw_n$ at hand. This is the pullback of the standard inner product
by the projection onto the subspace spanned by the first $n-1$ basis elements.

\begin{prop}
Let
$\alpha = \sum_j \ton_j \xi_j$,
$\beta = \sum_k \ttw_k \xi_k$,
$\gamma = \sum_l \tth_l \xi_l$
and $\delta = \sum_m \tfo_m \xi_m$
be holomorphic tangent fields close to $(0,0)$.
The curvature of $h_t$ at the origin is
$$
\displaylines{
R_{h_t}(\alpha, \ov\beta, \gamma, \ov\delta)
= H_h(e_n) \ton_n \bar \ttw_n \tth_n \bar \tfo_n
+ t \bigl(
\tau(\ton, \bar \ttw) \tau(\tth, \bar \tfo) + \tau(\ton, \bar \tfo) \tau(\tth, \bar \ttw)
\bigr)
\hfill\cr\hfill{}
- \ton_n \bar \ttw_n \tau(\tth, \bar \tfo)
- \ton_n \bar \tfo_n \tau(\tth, \bar \ttw)
- \tth_n \bar \ttw_n \tau(\ton, \bar \tfo)
- \tth_n \bar \tfo_n \tau(\ton, \bar \ttw).
}
$$
\end{prop}

\begin{proof}
Recall that the curvature of $h_t$ is
$$
R_{h_t}(\alpha, \ov\beta, \gamma, \ov\delta)
= p_1^* R_h(\alpha, \ov\beta, \gamma, \ov\delta)
+ t p_2^* R_b(\alpha, \ov\beta, \gamma, \ov\delta)
- \<\sigma(\alpha, \gamma), \ov{\sigma(\beta, \delta)}\>,
$$
where $\sigma(\alpha, \beta) = \pi_N(D_{h \oplus t \fs,\alpha} \beta)$ is
the second fundamental form, and the inner product is on the orthogonal complement
of the tangent bundle.

First,
$$
p_1^*R_h(\alpha, \ov\beta, \gamma, \ov\delta)
= H_h(e_n) \ton_n \bar \ttw_n \tth_n \bar \tfo_n
$$
at the origin. Second,
\begin{align*}
p_2^*R_b(\alpha, \ov\beta, \gamma, \ov\delta)
&= \< p_2(\alpha), \ov{p_2(\beta)} \>_b
\< p_2(\gamma), \ov{p_2(\delta)} \>_b
+ \< p_2(\alpha), \ov{p_2(\delta)} \>_b
\< p_2(\gamma), \ov{p_2(\beta)} \>_b
\\
&= \tau(\ton, \bar \ttw) \tau(\tth, \bar \tfo) + \tau(\ton, \bar \tfo) \tau(\tth, \bar \ttw)
\end{align*}
at the origin
because the Fubini--Study metric has constant \hsc{} 2
and $\tau$ is just the inner product on $\CC^{n-1}$ extended to $\CC^n$.

Let's now write $\pi_N$ for the projection onto the orthogonal complement of
$T_{\smash{\blX}}$ in $T_{\smash{\CC^n \times \CC^{n-1}|\blX}}$.
At the origin it is just the projection onto the first $n-1$ coordinates.
Third, then,
$$
\pi_N(D_{h_t,\alpha} \gamma)
= \pi_N\biggl(
\sum_{j=1}^{n} \tfo_{\alpha} \tth_j \, \xi_j
+ \sum_{j=1}^n \tth_j D_{\alpha} \xi_j
\biggr).
$$
Note that $\pi_N(\xi_j) = 0$ for $j=1,\ldots,n$ so the first term from that sum
is zero. The second term is
$$
\pi_N(D_{\alpha} \xi_j)
= \pi_N(D_{\alpha}(x_n e_j + f_j))
= \ton_n e_j + x_n D_{\alpha} e_j
= \ton_n e_j
$$
for $j=1,\ldots,n-1$ at the origin because $D_{\alpha} e_j = 0$ there.
We also have
$$
\pi_N(D_{\alpha} \xi_n)
= \pi_N\biggl(
\sum_{k=1}^{n-1} \tfo_{\alpha} y_k \, e_k + y_k D_{\alpha} e_k
+ D_{\alpha} e_n
\biggr)
= \sum_{k=1}^{n-1} \ton_k e_k
$$
at the origin.
Together we get
$$
\sum_{j=1}^n \tth_j \pi_N(D_{\alpha} \xi_j)
= \sum_{j=1}^{n-1} \tth_j \ton_n e_j + \tth_n \sum_{j=1}^{n-1} e_j
= \sum_{j=1}^{n-1} (\ton_n \tth_j + \tth_n \ton_j) e_j.
$$
And fourth,
$$
\pi_N(p_2^*D_{g,\alpha} \gamma)
= \pi_N \biggl(
\sum_{j=1}^{n-1} \tfo_\alpha \tth_j \, f_j
+ \tth_j p_2^*D_{g,\alpha} f_j \biggr) = 0
$$
at the origin because there the argument to $\pi_N$ takes values in the bundle
spanned by the $f_j$.
The second fundamental form then contributes
\begin{align*}
\<\sigma(\alpha, \gamma), \ov{\sigma(\beta, \delta)}\>
&=
\sum_{j=1}^{n-1} (\ton_n \tth_j + \tth_n \ton_j) (\bar \ttw_n \bar \tfo_j + \bar \tfo_n \bar \ttw_j)
\\
&= \ton_n \bar \ttw_n \tau(\tth, \bar \tfo)
+ \ton_n \bar \tfo_n \tau(\tth, \bar \ttw)
+ \tth_n \bar \ttw_n \tau(\ton, \bar \tfo)
+ \tth_n \bar \tfo_n \tau(\ton, \bar \ttw)
\end{align*}
to the curvature. Putting all of this together we get the result.
\end{proof}

A fairly natural strategy to proving that the blowup of a point of a manifold
that has a metric with positive \hsc{} is also positively curved is to
show that $h_t$ has positive \hsc{} on a neighborhood around the exceptional
divisor for all small enough $t$. Then we could conclude that the blowup has
positive \hsc{} by using Wu's~\cite{wu1973remark} characterization of the
\hsc{} to show positivity outside of that neighborhood for $t$ small enough.
The following corollary shows this strategy cannot work.

\begin{coro}
The holomorphic sectional curvature of $h_t$ is negative in
some tangent directions at the origin for all small enough $t$.
\end{coro}

\begin{proof}
Let $\xi = \sum_j \ton_j \xi_j$ be a holomorphic tangent field.
Note that $|\xi_j|_{h_t}^2 = t$ for $j = 1, \ldots, n-1$ and $|\xi_n|_{h_t}^2 = 1$
at the origin.
We have
$$
H_{h_t}(\xi)
= \frac{|\ton_n|^4 H_h(e_n)
+ 2t \tau(\ton, \bar \ton)^2
- 4 |\ton_n|^2 \tau(\ton, \bar \ton)}{(t \tau(\ton, \bar \ton) + |\ton_n|^2)^2}.
$$
First suppose that $H_h(e_n) < 0$. Picking $\ton$ with $|\ton_n| = 1$ we have
$\tau(\ton, \bar \ton) = 0$ so $H_{h_t}(\xi) = H_h(e_n) < 0$.

Then suppose that $H_h(e_n) \geq 0$.
We may assume that $|\ton| = 1$ when checking for positivity and focus only on
the numerator. We're then left considering the positivity of the polynomial
\begin{align*}
p_t(x)
&= H_h(e_n) x^2 + 2t(1-x)^2 - 4x(1-x)
\\
&= (H_h(e_n) + 2t + 4) x^2 - 4(1 - t) x + 2t
\end{align*}
for $0 \leq x \leq 1$.
Its critical point is
$$
x_t = \frac{2(1-t)}{H_h(e_n) + 2t + 4}
$$
where its value is
$$
p_t(x_t) = -\frac{4(1-t)^2}{H_h(e_n) + 2t + 4} + 2t,
$$
which is negative for all small $t$.
\end{proof}

This corollary shows that \emph{no} metric of the form $\mu^* h + tb$ can have
positive \hsc{} when $h$ does, so we can't prove that the blowup carries such
a metric in this way.
Note however that the blowup of the projective plane at a single point
is a Hirzebruch surface, thus a projective bundle, and thus carries
a metric of positive \hsc{}.

\begin{coro}
Let
$\alpha = \sum_j \ton_j \xi_j$,
$\beta = \sum_k \ttw_k \xi_k$
be holomorphic tangent fields close to $(0,0)$.
The Ricci tensor of $h_t$ is
$$
r_t(\alpha, \ov\beta)
= H_h(e_n) \ton_n \bar \ttw_n
+ (n-1) \tau(\ton, \bar \ttw)
- (n-1) \ton_n \bar \ttw_n / t
$$
at the origin.
\end{coro}

\begin{proof}
The tangent fields that correspond to
$v_j = (\delta_{1j}/\sqrt t, \ldots, \delta_{n-1,j}/\sqrt t, \delta_{nj})$ are
orthonormal at the origin.
If $\xi$ and $\eta$ are the fields corresponding to $(\ton_j)$ and $(\ttw_j)$ we
first have
$$
\sum_{j=1}^n H_h(e_n) \ton_n \bar \ttw_n \delta_{jn} = H_h(e_n) \ton_n \bar \ttw_n,
$$
and then
$$
\sum_{j=1}^n
\tau(\ton, \bar \ttw) \tau(v_j, \bar v_j) + \tau(\ton, \bar v_j) \tau(v_j, \bar \ttw)
= n \tau(\ton, \bar \ttw) / t.
$$
The second fundamental form contributes
$$
\displaylines{
\sum_{j=1}^n
\ton_n \bar \ttw_n \tau(v_j, \bar v_j)
+ \ton_n \delta_{jn} \tau(v_j, \bar \ttw)
+ \delta_{jn} \bar \ttw_n \tau(\ton, \bar v_j)
+ \delta_{jn}^2 \tau(\ton, \bar \ttw)
\hfill\cr\noalign{\vskip-10pt}\hfill{}
= (n-1) \ton_n \bar \ttw_n / t + \tau(\ton, \bar \ttw)
}
$$
because $\tau(\ton, \bar v_n) = 0$.
Together we get
\[
r_t(\ton, \bar \ttw)
= H_h(e_n) \ton_n \bar \ttw_n
+ (n-1) \tau(\ton, \bar \ttw)
- (n-1) \ton_n \bar \ttw_n / t.
\qedhere
\]
\end{proof}

We see that the Ricci tensor can also be negative in some directions for $t$
large enough, but is positive in other directions for all $t$.

\begin{coro}
The scalar curvature of $h_t$ is
$$
s_t
= H_h(e_n)
+ (n-1)(n-2) / t
$$
at the origin.
\end{coro}

This is the same as Hitchin gets for the scalar curvature if we take his
expressions for $R_0$ and $R_1$ in \cite[Lemmas~5.15 and
5.16]{hitchin1975curvature} and evaluate them at $r = 0$, which is
a nice sanity check.
With more work I think we could recover Hitchin's full results.
We've made some attempts in that direction to see if we could extend his
result on positive scalar curvature to Hermitian metrics, but seem to run
into obstructions due to the nonvanishing of the connection at the origin.

\bibliographystyle{plainurl}
\bibliography{main}

\end{document}